\documentclass[12pt,reqno]{amsart}
\usepackage{geometry}
\usepackage[latin1]{inputenc}
\usepackage[italian,english]{babel}
\usepackage{amsmath, amsfonts, amsthm,xcolor}
\usepackage{paralist}
\numberwithin{equation}{section}
\usepackage{mathtools, mathabx}
\newtheorem{thm}{Theorem}[section]

\newtheorem{lem}[thm]{Lemma}

\newtheorem{prop}[thm]{Proposition}

\newtheorem{defn}[thm]{Definition}
\theoremstyle{definition}

\theoremstyle{remark}

\newcommand{\ds}{\displaystyle}

\newcommand{\R}{\mathbb{R}}
\newcommand{\N}{\mathbb{N}}

\newcommand{\de}{\partial}
\newcommand{\eps}{\varepsilon}

\DeclareMathOperator{\spt}{spt}
{\left\{\begin{array}{@{}l@{}}}{\end{array}\right.}
\patchcmd{\abstract}{\scshape\abstractname}{\textbf{\abstractname}}{}{}
\makeatletter %note a di pagina senza numero 1
\def\@makefnmark{} %note a di pagina senza numero 2
\makeatother %note a di pagina senza numero 3

\usepackage[pagewise]{lineno}%\linenumbers
\title{A stability result for the Steklov Laplacian Eigenvalue Problem with a spherical obstacle}
\author[G. Paoli, G. Piscitelli,  R. Sannipoli]{Gloria Paoli$^1$, Gianpaolo Piscitelli$^2$, Rossano Sannipoli$^1$}
\begin{document}
\maketitle
\noindent{\footnotesize{
\emph{$^1$Dipartimento di Matematica e Applicazioni ``R. Caccioppoli'', Universit\`a degli studi di Napoli Federico II, Via Cintia, Complesso Universitario Monte S. Angelo, 80126 Napoli, Italy.\\
$^2$Dipartimento di Ingegneria Elettrica e dell'Informazione \lq\lq M. Scarano\rq\rq, Universit\`a degli Studi di Cassino e del Lazio Meridionale, Via G. Di Biasio n. 43, 03043 Cassino (FR), Italy.}}\\
{\rm E-mail: gloria.paoli@unina.it, gianpaolo.piscitelli@unicas.it {\it(corresponding author)},\\ rossano.sannipoli@unina.it.}}
\markright{A STABILITY RESULT FOR THE STEKLOV LAPLACIAN EIGENVALUE}

\begin{abstract}
In this paper we study the first Steklov-Laplacian eigenvalue with an internal fixed spherical obstacle. We prove that the spherical shell locally maximizes the first eigenvalue among nearly spherical sets when both the internal ball and the volume are fixed.\\

\end{abstract}

\section{Introduction}
Let $\Omega_0\subset\mathbb{R}^n$, $n\geq 2$, be an open, bounded, connected set, with Lipschitz boundary such that $B_{r}\Subset\Omega_0$, where $B_{r}$ is the open ball of radius $r>0$ centered at the origin. Let us set $\Omega :=\Omega_0\setminus \overline{B_{r}}$, then we study  the following Steklov-Dirichlet boundary eigenvalue problem for the Laplacian:
\begin{equation}\label{case1intro}
\begin{cases}
\Delta u=0 & \mbox{in}\ \Omega\\
u=0&\mbox{on}\ \partial B_{r}, \\
{\de_\nu u}=\sigma (\Omega)u&\mbox{on}\ \partial\Omega_0\\ 
\end{cases}
\end{equation}
where  $\nu $ is the outer unit normal to  $\partial\Omega_0$. The study of the first eigenvalue of problem \eqref{case1intro} leads to the following minimization problem:
\begin{equation}\label{minSD_intro}
\sigma_1( \Omega)=\min_{\substack{w\in H^{1}_{\partial B_{r}}(\Omega)\\ w\not \equiv0} }     
\dfrac{\ds\int _{\Omega}|Dw|^2\;dx}{\ds\int_{\partial\Omega_0}w^2\;d\mathcal{H}^{n-1}}\;,
\end{equation}
where $H^{1}_{\partial B_{r}}(\Omega)$ is the set of Sobolev functions on $\Omega$ that vanish on $\partial B_{r}$ (for the precise definition see Section \ref{prel_sec}).  Notice also that the value $\sigma_1(\Omega)$ is the optimal constant in the Sobolev-Poincar\'e trace inequality:
\begin{equation}\label{traceineq}
\sigma_1(\Omega)||w||_{L^2(\partial\Omega_0)}\leq ||Dw||_{H^1_{\partial B_{r}}(\Omega)}.
\end{equation}
In this paper we treat the following shape optimization issue: 
\vspace*{2mm}
\begin{center}
{\it Which sets maximize $\sigma_1(\cdot)$ among sets containing the fixed ball $B_{r}$ and having prescribed
measure?}
\end{center}
\vspace*{2mm}
We  partially solve the problem of the optimality of $\sigma_1$, restricting our study to nearly spherical sets, that are set whose boundary can be parametrized on the sphere by means of a Lipschitz function with a small $W^{1,\infty}$-norm. The main result of the paper is the following.

\vspace{5pt}
{\bf Main Theorem.} {\it
Let $\Omega =\Omega_0\setminus\overline{B_{r}}$, with $\Omega_0$ a nearly spherical set. Then 
\begin{equation}\label{main_ineq_intro}
\sigma_1 (\Omega)\leq \sigma_1(A_{r,R}),
\end{equation}
where $A_{r,R}=B_{R}\setminus\overline{B_{r}}$, with $R>r>0$, is the spherical shell with the same volume as $\Omega$. Moreover the equality in \eqref{main_ineq_intro} holds if and only if $\Omega$ is a spherical shell.}
\vspace{5pt}

So, we study the optimal shape for $\sigma_1(\Omega)$ when both the volume of the domain and the radius of the internal ball are fixed. We also find some counterexamples showing that when only a volume constraint holds, then $\sigma_1$ is not upper bounded, hence we cannot speak about optimality. 

In order to prove the main Theorem, we obtain a stability result in quantitative form. In Theorem \ref{imp}, we find $K=K(n, |\Omega|)>0$, such that
\begin{equation*}
\sigma_1(A_{r,R}) \ge  \sigma_1(\Omega) \left( 1+K(n,|\Omega|) \int_{\mathbb{S}^{n-1}} v^2(\xi)\,d\mathcal{H}^{n-1}\right).
\end{equation*}
When $r=0$ and $\Omega_0$ is connected, the problem becomes the Steklov eigenvalue problem introduced by Steklov.  Optimal upper bounds for classical Steklov eigenvalues have been proved by several authors. When the domain is simply connected, a Weinstock inequality (\cite{W} for $n=2$ and \cite{BFNT} for higher dimensions) holds. This means that, among convex sets with prescribed perimeter, the maximum for the first Steklov Laplacian eigenvalue is reached by the ball. % (see \cite{GPT} for the quantitative version). 
On the other hand, in \cite{B}, the author proved that the ball is the maximum for the same eigenvalue keeping the volume fixed. %(see \cite{brasco2012spectral} for the quantitative version).

In the class of sets of the form $B_R(x_0)\setminus \overline B_r$ with $B_R(x_0)$ being a ball containing in $B_r$, the maximizer of $\sigma_1$ is the spherical shell, that is the annulus when the balls are concentric (see \cite{Ft,SV}).% This is also proved in \cite{SV} and more for general spaces in \cite{seo2019shape}.
%the problem of optimally insulating a given domain \cite{DNT}

A natural way to prove the result in \eqref{main_ineq_intro} is in finding the right test function for the Raylegh quotient in \eqref{minSD_intro}, in order to obtain the sought spectral inequality.
This approach works for other
%Unfortunately in general (9) is false. For instance, for what we said before the function z(x-x0) is the unique extremal in 1 when $\Omega=B_r(x_0)$ and in this case we have \[\frac{V(B_r(x_0))}{P(B_r(x_0))}>\sigma(B_r(x_0))=\sigma(B_r(0))=\frac{V(B_r(0))}{P(B_r(0))}.\]
 mixed boundary condition eigenvalue problems on perforated domains, e.g when using the so-called web functions (see \cite{CFG} % \cite{bucur2017weinstock, brandolini2010upper, crasta2002sharp}
 and the references  therein). 
 In particular, in \cite{PPT} it is proved that  the first eigenvalue of the $p$-Laplacian  with  external Robin and internal Neumann boundary conditions is maximum on spherical shells,  when the  volume and the  external perimeter are fixed. In  \cite{PW} can be found the  original proof for $p=2$ in the bidimensional case.   In \cite{DP} the authors prove that the first eigenvalue of the $p$-Laplacian with external Neumann and internal Robin boundary conditions is maximum on spherical shells  when the volume and the  internal $(n-1)$-quermassintegral are fixed. See \cite{H} for the original  proof in the plane and for $p=2$.% This result recovers the Theorem proved for the Laplacian case in \cite{H} for the bidimensional case. %Furthermore, similar estimates has been obtained also for a more general class of equations, involving the so called Finsler operator. We refer the reader, for example, to \cite{della2014faber,della2019second,della2017sharp,gavitone2018first,paoli2019two,piscitelli2019anisotropic}.

We use the solution $z$ of \eqref{minSD_intro} on the spherical shell to introduce the \textit{weighted volume} $V(\Omega)$ and the \textit{weighted perimeter} $P(\Omega)$:
\begin{equation*}
\begin{split}
V(\Omega)&:= \int_{\Omega} |\nabla z|^2 \,dx,\\
P(\Omega)&:= \int_{\partial\Omega_0} z^2 \,dx.
\end{split}
\end{equation*}
So we have
\[
\sigma_1(\Omega)\leq\frac{V(\Omega)}{P(\Omega)},
\]
with equality at least in the case $\Omega=A_{r,R}$, where $|\Omega|=|A_{r,R}|$. 
Unfortunately, using the web-function testing method in the Rayleigh quotient  \eqref{minSD_intro}, we do not obtain the correct inequality:
\begin{equation}\label{uppRay}
\frac{V(\Omega)}{P(\Omega)}\leq\frac{V(A_{r,R})}{P(A_{r,R})}.
\end{equation}
%satisfying %that is 
%$$ \de\Omega_0=\{ y\in\mathbb{R}^n\;|\; y=\rho(1+v(\xi)),\;\xi\in\mathbb{S}^{n-1}  \},$$
%where $\rho$ is the radius of the ball with the same measure as $\Omega$ and $v\in W^{1,\infty}(\Omega)$ such that  $||v||_{W^{1,\infty}}<\varepsilon $, $$   B_{r}\subset   \{  x\in\mathbb{R}^n\;:\; {\rm dist}(x, \mathbb{R}^n\setminus B_{R} ) >\delta  \}\subset  \Omega_0\subset \{  x\in\mathbb{R}^n\;:\; {\rm dist}(x, B_{R} ) <\delta  \}$$  for some positive $\delta $, that depends on the volume of $\Omega$ and $n$ only. % equality holds iff \Omega is an annulus
%n this paper we give a local upper bound for $\sigma_1(\Omega)$ among open bounded domain holed by a ball when volume is fixed.
%Hence, we do not consider separately $V(\cdot)$ and $P(\cdot)$. 
In this paper, we restrict our study to the class of sets $\Omega =\Omega_0\setminus\overline{B_{r}}$, with $\Omega_0$ nearly spherical. We use some classical stability results for isoperimetric problems (see e.g. 
\cite{Fu})
%\cite{fuglede1989stability, fusco2015quantitative, fusco2017stability, esposito2005quantitative})
to parametrize the outer boundary of $\Omega$ and then we perform a Taylor expansion (as in \cite{FNT}) to translate the sought spectral inequality into the Poincar\'e inequality \eqref{traceineq}.

The outline of the paper follows. In Section \ref{prel_sec} we give some properties on the mixed Steklov-Dirichlet eigenvalue problem we are dealing with. In Section \ref{main_sec}, we prove the main Theorem.

\section{The Eigenvalue Problem}\label{prel_sec}
Let $R>r>0$, throughout this paper, we denote $B_{r}:=\{ x\in\R^n\ : \ |x|<r\}$ the ball centered at the origin with radius $r>0$; $A_{r,R}$ the spherical shell $B_R\setminus \overline{B}_r$ and
\begin{equation*}
\mathcal{A}_r := \left\{
\begin{split}
\Omega =\Omega_0\setminus \overline{B_{r}}\ : \, \Omega_0\subset\R^n \, & \text{open, bounded, connected,} \\
& \text{with Lipschitz boundary}, \, \text{s.t.} B_{r}\Subset\Omega_0
    \end{split}\right\}.
\end{equation*} %Let us remark that $\Omega_0$ can have more than one holes. 
Furthermore, we denote by $\mathcal H^{n-1}$ the $(n-1)$-dimensional Haussdorf measure and  by $|\cdot|$ the Lebesgue measure in $\mathbb{R}^n$.

Since we are  studying a Steklov eigenvalue problem with a spherical obstacle, we need to introduce the definition of a closed subspace of $H^1(\Omega)$,  that incorporates the Dirichlet boundary condition on $\partial{B_{r}}$. We denote the set of Sobolev functions on $\Omega$ that vanish on $\partial B_{r}$ by
$$H^1_{\partial B_{r}}(\Omega),$$ 
that is (see \cite{ET}) the closure in $H^1(\Omega)$ of the set of test functions
\begin{equation*}%\label{sobolev_part}
C^\infty_{\partial B_{r}} (\Omega):=\{ u|_{\Omega}  \ | \ u \in C_0^\infty (\R^n),\ \spt (u)\cap \partial B_{r}=\emptyset \}. 
\end{equation*}

\subsection{Eigenvalues and Eigenfunctions}
We are dealing with the following boundary eigenvalue problem:
\begin{equation}\label{eigSD}
\begin{cases}
\Delta u=0 & \mbox{in}\ \Omega\\
u=0&\mbox{on}\ \partial B_{r},\\ 
\de_\nu u=\sigma(\Omega) u&\mbox{on}\ \partial\Omega_0
\end{cases}
\end{equation}
where $\nu $ is the outer normal to $\partial \Omega_0$. We give now the definitions and some geometric properties of eigenvalues and eigenfunctions of problem \eqref{eigSD}.
\begin{defn}
The real number $\sigma(\Omega)$ and the function $u\in H^1_{\partial B_{r}}(\Omega)$ are, respectively, called eigenvalue of \eqref{eigSD} and eigenfunction associated to $\sigma(\Omega)$, if and only if 
	\begin{equation*}%\label{weakly}
	\int_{\Omega} D uD\varphi \;dx=\sigma(\Omega)\int_{\partial\Omega_0}u \varphi \;d\mathcal{H}^{n-1},
	\end{equation*}
	for every $\varphi\in H^1_{\partial B_{r}}(\Omega)$.
\end{defn}
Furthermore, the first eigenvalue is variationally characterized by 
\begin{equation}\label{minSD}
\sigma_1( \Omega)=\min_{\substack{ w\in H^1_{\partial B_{r}}\\ w\not \equiv0}(\Omega) }      J[w],\,%:=\min_{\substack{w\in W^{1,p}(\Omega)\\ w\not \equiv0}} \dfrac{\ds\int _{\Omega}|Dw|^p\;dx+\beta\ds\int_{\Gamma_0}|w|^p \;d\mathcal{H}^{n-1}}{\ds\int_{\Omega}|w|^p\;dx}, 
\end{equation}
where
\begin{equation}
\label{ray_quo}
J[w]:=\dfrac{\ds\int _{\Omega}|Dw|^2\;dx}{\ds\int_{\partial\Omega_0}w^2\;d\mathcal{H}^{n-1}}.
\end{equation}
We point out that the condition of being orthogonal to constants in $L^2(\partial\Omega)$ is not required, unlike the classical Steklov eigenvalue (when $r=0$).%, where the normalized eigenfunctions give an orthonormal basis of $L^2(\partial\Omega)$.

The following ensures the existence of  minimizers of problem \eqref{minSD}. 
\begin{prop}\label{existence_prop}
Let $r>0$ and $\Omega\in\mathcal{A}_r$, then there exists a function $u\in H^1_{\partial B_{r}}(\Omega)$ achieving  the minimum in \eqref{minSD} and satisfying problem \eqref{eigSD}. Moreover, $u$ is positive (or negative) in $\Omega$.
\end{prop}
\begin{proof}
Let $u_k\in H^1_{\partial B_{r}}(\Omega)$ be a minimizing sequence of \eqref{minSD} such that $||u_k||_{L^2(\partial\Omega_0)}=1$. Since the minimum in \eqref{minSD} is positive, then there exists a constant $C>0$ such that $J[u_k]\leq C$ for every $ k\in \mathbb{N}$ and therefore $||Du_k||_{L^2(\Omega)}\leq {C}$. Moreover, a  Poincar\'e inequality in $H^1_{\partial B_{r}}(\Omega)$ holds and this implies that  $\{ u_k \}_{k\in\N}$ is a bounded sequence in $H^1_{\partial B_{r}}(\Omega)$.
	%	Let $u_k\in H^{1}(\Omega)$ be a minimizing sequence of \eqref{minSD} such that $||u_k||_{L^2(\partial\Omega_0)}=1$. By \eqref{trivial_up}, there exists a constant $C>0$ such that $J[u_k]\leq C$ for every $ k\in \mathbb{N}$, hence $\{ u_k \}_{k\in\N}$ is a bounded sequence in $H^1(\Omega)$. 
Therefore, there exist a subsequence, still denoted by $u_k$, and a function $u\in H^1_{\partial B_{r}}(\Omega)$ with $||u||_{L^2(\partial\Omega_0)}=1$, such that $u_k\to u$ strongly in $L^2(\Omega)$, hence also almost everywhere, and $D u_k\rightharpoonup D u$ weakly in $L^2(\Omega)$. By the compactness of the trace operator,% (see for example \cite[Cor. 18.4]{leoni2017first}),
\phantom{ }$u_k$ converges strongly  to $u$ in $L^2(\de \Omega)$ and almost everywhere on $\de \Omega$ to  $u$. Then, by weak lower semicontinuity we have
\begin{equation*}
	\lim\limits_{k\to+\infty }J[u_k]\geq J[u].
\end{equation*}
Hence the existence of a minimizer $u\in H^1_{\partial B_{r}}(\Omega)$ follows.
%Moreover, $u$ is harmonic in $\Omega$ and hence, by strong maximum principle, it has constant sign on $\Omega$. 
Moreover, the fact that $$J[u]=J[|u|]$$ implies that any eigenfunction must have constant sign on $\Omega$. So, by Harnack inequality, %(see \cite[Th 1.1]{trudinger1967harnack}),
$u$ is strictly positive on $\Omega$.
\end{proof} 
Now we state the simplicity of the first eigenvalue of \eqref{eigSD}.%, following the idea in \cite[6.5.1]{lawrence2010evanspartial}.
\begin{prop}%\label{sign}
Let $r>0$ and $\Omega\in\mathcal{A}_r$, then the first eigenvalue $\sigma_1(\Omega)$ of \eqref{eigSD} is simple, that is all the associated eigenfunctions are scalar multiple of each other.
\end{prop}
\begin{proof}
Let $u,\tilde{u}$ be two non trivial weak solutions of the problem (\ref{eigSD}). Since, by Proposition \ref{existence_prop}, we can assume that $\tilde{u}$ is positive in $\Omega$, then it is clear that
\begin{equation*}
    \int_{\Omega} \tilde{u} \,dx \neq 0.
\end{equation*}
So, we can find a real constant $\chi$ such that
\begin{equation} \label{nullmean}
    \int_{\Omega} (u-\chi \tilde{u}) \,dx = 0.
\end{equation}
Since $u-\chi \tilde{u}$  is still a solution of the problem (\ref{eigSD}), then it is also non-negative (or non-positive) in $\Omega$. Therefore, (\ref{nullmean}) implies that $u \equiv \chi \tilde{u}$ in $\Omega$ and the simplicity of $\sigma_1(\Omega)$ follows. 
\end{proof}
It is worth noticing that the first nontrivial eigenvalue for the classical Steklov-Laplacian problem (when $r=0$) on $B_R$ is $1/R$ and the corresponding eigenfunctions are the coordinate axis $x_i$, for $i=1,..,N$. This means that the first nontrivial eigenvalue has multiplicity $N$ and this makes a huge difference with problem \eqref{eigSD}, for which we proved that the simplicity holds.

On the other hand, it is easy to verify that both have the same scaling property:
\begin{equation}\label{scaling}
    \sigma(t\Omega)=\frac 1 t \sigma(\Omega), \quad\forall t\in\R.
\end{equation}

The first attempts to study the optimal shape of problem \eqref{eigSD} has been done on spherical shells, i.e. when $\Omega_0=B_R$, for $R>r>0$. We recall from \cite{SV}, the explicit expression of the first eigenfunction on the spherical shell $A_{r,R}$:	
\begin{equation}\label{radeigfun}
	z(\rho)=\begin{cases}
	\ln \rho-\ln r
	& {\rm for}\;\; n=2\vspace{0.1cm}\\
	\left( \dfrac{1}{r^{n-2}}-\dfrac{1}{\rho^{n-2}}\right)& {\rm for}\;\; n\geq 3\vspace{0.1cm},
	\end{cases}
	\end{equation}
	with $\rho=|x|$. This function is radial, positive, strictly increasing and it is associated to the following eigenvalue:
\begin{equation}\label{radeigval}
\sigma_1(A_{r,R})=
\begin{cases}
\frac{1}{R\log\left(\frac{R}{r}\right)}& {\rm for}\;\; n=2\vspace{0.1cm}\\
\frac{n-2}{R\left[\left(\frac{R}{r}\right)^{n-2}-1\right]}& {\rm for}\;\; n\geq 3\vspace{0.1cm}.\\ 
\end{cases}
\end{equation}
It is worth noting that, since problem \eqref{eigSD} and the classical Steklov ($r=0$) have the same scaling property \eqref{scaling}, then the shape functional $\Omega \to |\Omega|^\frac{1}N\sigma(\Omega)$ is scaling invariant, as in the classical case.

\subsection{A first upper bound}
We show an upper bound for $\sigma_1$ depending only by the dimension $n$, the measure of $\Omega$ and by the radius of the internal ball $r$.
\begin{prop} Let $r>0$ and $\Omega\in\mathcal{A}_r$, then% the first eigenvalue of \eqref{eigSD} is upper bounded bythere exists a constant only depending by $r$ and the dimension, such that
\begin{equation*}%\label{simple_upper}
\sigma_1(\Omega)\leq \dfrac{2}{n\omega_n^{\frac 1n}\left(\left(   \dfrac{|\Omega|}{2\omega_n}+r^n\right)^{1/n}-r\right)^2}|\Omega|^{1/n}.
\end{equation*}
\end{prop}
\begin{proof}
Let $\bar{R}>0$ be such that $|A_{r,\bar{R}}|=|\Omega|/2$, then $\bar{R}$ depends only by the dimension $n$, the measure $|\Omega|$ and $r$, that is
\begin{equation*}%\label{explicit_radius}
	\bar{R}=\left(   \dfrac{|\Omega|}{2\omega_n}+r^n\right)^{1/n}.
\end{equation*}
Consider the function 
\begin{equation}\label{test_upper_bound}
\varphi(x)=\begin{cases}
|x|-r & \mbox{if  } \; r\leq|x| \leq \bar{R};  \vspace{0.2cm}\\
\bar{R}-r & \mbox{if  } \; |x| \geq \bar{R}.  \vspace{0.2cm} 
\end{cases}
\end{equation}
We distinguish now two cases. Firstly, we assume that $B_{\bar{R}}\Subset\Omega_0$, i.e. $d:={\rm dist}(\de B_{\bar{R}},\de\Omega_0 )>0$.
By using \eqref{test_upper_bound} as test function in the Rayleigh quotient \eqref{uppRay} and by the isoperimetric inequality, we obtain
\begin{equation}
\label{upper_dentro}
\sigma_1(\Omega)\leq \dfrac{|\Omega|}{\left( \bar{R}-r  \right)^2 P(\Omega_0)}\leq \dfrac{1}{n \omega_n^{\frac 1n} \left( \bar{R}-r  \right)^2}|\Omega|^{\frac 1n}.
\end{equation}
We consider now  the case $d=0$, that is when the ball $B_{\bar{R}}$ is not strictly contained in $\Omega_0$. Therefore, we divide the boundary of $\Omega_0$ in the two sets $\partial^{int}\Omega_0$ and $\partial^{ext}\Omega_0$ that live, respectively, inside and outside of $B_{\bar R}$. Using  the test function \eqref{test_upper_bound} in the Raylegh quotient \eqref{ray_quo}, we have
\begin{equation}\label{distance_zero}
	\sigma_1(\Omega)\leq \dfrac{|\Omega|}{\int_{\partial\Omega_0} |\varphi|^2\;d\mathcal{H}^{n-1}   }\leq \dfrac{|\Omega|}{(\bar{R}-r)^2 \int_{\de^{ext}\Omega_{0}}   1\; d\mathcal{H}^{n-1} }.
\end{equation} 
%where $\de\Omega_{0,e}:=\de\Omega_0\setminus\left(   \de\Omega_0\cap B_{\bar{R}} \right) $. 
We recall that a relative isoperimetric inequality with supporting set $B_{\bar{R}}$ holds (see as a reference e.g. \cite{CGR}):
\begin{equation}\label{relative_isoperimetric}
\mathcal{H}^{n-1}(\de^{ext}\Omega_{0})\geq n \left(\dfrac{\omega_n}{2}\right)^{1/n} \left(  \frac{|\Omega_0|}{2} \right)^{1- \frac 1n}. 
\end{equation}
By using \eqref{relative_isoperimetric} in \eqref{distance_zero}, we have
\begin{equation}
\label{upper_fuori}
\sigma_1(\Omega)\leq \dfrac{2}{n\omega_n^{\frac 1n}(\bar{R}-r)^2}|\Omega|^\frac 1n.
\end{equation}
The conclusion follows by observing that the upper bound \eqref{upper_fuori} is greater than \eqref{upper_dentro}.
\end{proof}
%We remark that this proof gives the expression of the upper-bounding constant \eqref{constant}. 
We remark that, when a volume constraint for $\Omega$ holds, then the upper bound is still finite, when $r\to 0$. On the other hand, when $r\to\infty$, the first eigenvalue cannot be upper bounded. This, together with other examples we are giving in the rest of this Section, motivates the study the optimality of $\sigma_1$ when another constraint holds, besides the volume one.\\
%We stress that by testing $w(x)=|x|-r$ in the functional in \eqref{functional} and then using an isoperimetric inequality, we have a trivial upper bound for $\sigma_1$, only depending by the geometry (volume, inradius $r_\Omega$ and radius of internal ball $r$) of $\Omega$:\begin{equation*}\sigma_1(\Omega)\le \frac{ |\Omega|}{(r_{\Omega} -r)^2P(\partial\Omega_0)}\leq \frac{ |\Omega|}{n\omega_n^\frac1n(r_{\Omega} -r)^2 |\Omega_0|^{1-\frac 1n}}\leq \frac{|\Omega|^\frac 1n}{n\omega_n^\frac1n(r_{\Omega} -r)^2 }.
%\end{equation*} where $r_{\Omega}$ is the inradius of $\Omega$.

\subsection{Volume constraint on the spherical shells}%\label{onlyvol}
In this paper we deal with geometric properties of the first eigenvalue of \eqref{eigSD}. We look for shapes minimizing $\sigma_1(\Omega)$, when both $\omega$, the volume of $\Omega$, and the radius $r$ of the internal ball are fixed. We show that, even among the spherical shells, $\sigma_1$ cannot be upper bounded when only a volume constraint holds. 

Let us consider the spherical shell $A_{r,R}$ with the volume constraint:
\begin{equation*}
	|A_{r,R}| = \omega_n (R^n-r^n)= \omega.
\end{equation*}
We show that both in bidimensional case and in higher dimension, $\sigma_1$ is not upper bounded in the class of sherical shells of fixed volume.\\	%{\bf Case 1.} 
Let $n=2$, then $R = \left(r^2 + \frac{\omega}{\pi}\right)^{\frac{1}{2}}$ and, by \eqref{radeigval}, we have
	\begin{equation*}
		\sigma_1(A_{r,R})= \frac{1}{\left(r^2 + \frac{\omega}{\pi}\right)^{\frac{1}{2}} \log\left(1 + \frac{\omega}{\pi r^2}\right)^{\frac{1}{2}}} = \frac{2}{r \left(1 + \frac{\omega}{\pi r^2}\right)^{\frac{1}{2}} \log\left(1 + \frac{\omega}{\pi r^2}\right)}.
	\end{equation*}
	Hence for $r$ big enough:
	\begin{equation*}
	\sigma_1(A_{r,R})\approx \frac{2}{r \left(1 + \frac{\omega}{2\pi r^2}\right)  \frac{\omega}{\pi r^2}} = \frac{2\pi r}{\omega \left(1 + \frac{\omega}{2\pi r^2}\right) },
	\end{equation*}
	and so 
	\begin{equation*}
		\lim_{r \to +\infty} \sigma_1(A_{r,R}) = +\infty.
	\end{equation*}\\%	{\bf Case 2.} 
Let $n\ge 3$, then $R = \left(r^n + \frac{\omega}{\omega_n}\right)^{\frac{1}{n}}$ and
%	\begin{equation*}
	\begin{align*}
	\sigma_1(A_{r,R}) &= \frac{n-2}{r \left(1 + \frac{\omega}{\omega_n r^n}\right)^{\frac{1}{n}}\left[\left(1 + \frac{\omega}{\omega_n r^n}\right)^{1-\frac{2}{n}}-1\right]} =\\
	&=\frac{n-2}{r \left[\left(1 + \frac{\omega}{\omega_n r^n}\right)^{1-\frac{1}{n}}-\left(1 + \frac{\omega}{\omega_n r^n}\right)^{\frac{1}{n}}\right]}.
	\end{align*}
%	\end{equation*}
	Again, if $r$ is big
	\begin{equation*}
	\sigma_1(A_{r,R}) \approx \frac{n-2}{r \left[1+\left(1 -\frac{1}{n}\right) \frac{\omega}{\omega_n r^n}-1 -\frac{1}{n} \frac{\omega}{\omega_n r^n}\right]} = \frac{n \omega_n}{\omega} r^{n-1}.
	\end{equation*}
	and hence again
	\begin{equation}\label{r1toinfty}
	\lim_{r \to +\infty} \sigma_1(A_{r,R}) = +\infty.
	\end{equation}
Further, it is clear that, in any dimension, we have
\begin{equation}\label{r1to0}
\lim_{r \to 0^+} \sigma_1(A_{r,R}) = 0.
\end{equation}
The limiting results \eqref{r1toinfty} and \eqref{r1to0} motivate the fact that it is not sufficient to fix the volume to study the first eigenvalue $\sigma_1$. Indeed, when $r$ is too big, it is not possible to find an upper bound, and, on the other hand, when $r$ is too small, the eigenvalue is trivial. We remark that, in the class of sets of the form $B_R(x_0)\setminus \overline B_r$ with $B_R(x_0)$ being a ball containing in $B_r$, the maximizer of $\sigma_1$ is the spherical shell (see \cite{Ft}).

\subsection{Spherical shell with fixed difference between radii.}%\label{diffrad}
It is clear now that we cannot study the shape optimization for $\sigma_1$ when only a volume constraint holds. On the other hand, it could be interesting to understand if we can study the shape optimization for double connected domains, when only one geometric quantity is fixed. Here, e.g., we briefly study the behavior of the spherical shell when the distance between the radii is fixed. 
Let $d$ be a positive real number such that
\begin{equation*}
	R-r = d,
\end{equation*}
so that $R = r + d$ and $\frac{R}{r} = 1 + \frac{d}{r}$.\\%	{\bf Case 1.} 
If $n=2$, then for $r$ big enough, we have
	\begin{equation*}
	\sigma_1(A_{r,R})= \frac{1}{(r + d) \log \left( 1 + \frac{d}{r}\right)}\approx \frac{r}{r d + d^2},
	\end{equation*}
	and hence
	\begin{equation*}
	\lim_{r \to +\infty} \sigma_1(A_{r,R}) = \frac{1}{d}.
	\end{equation*}\\%	{\bf Case 2.} 
If $n\ge 3$, we have
	%\begin{equation*}
	\begin{align*}
	\sigma_1(A_{r,R}) &= \frac{n-2}{(r + d) \left[ \left(1+\frac{d}{r}\right)^{n-2}-1\right]}\\
	& \approx \frac{n-2}{(r + d) \left[ 1+(n-2)\frac{d}{r}-1\right]} = \frac{r}{r d + d^2},
	\end{align*}
   % \end{equation*}
and hence
	\begin{equation*}
	\lim_{r \to +\infty} \sigma_1(A_{r,R}) = \frac{1}{d}.
	\end{equation*}
Furthermore, in any dimensions, we have
\begin{equation*}
	\lim_{r \to 0^+} \sigma_1(A_{r,R}) = 0
\end{equation*}
The case of  $r$ small is again trivial. On the other hand, $\sigma_1$ is upper bounded for any value of $R$ by the reciprocal of the difference between the radii $d$. The fact that a uniform upper bounds holds for spherical shells when only the difference between the radii is fixed, suggests that could be interesting to study the shapes minizing $\sigma_1$ in the class of double connected sets when only the width is fixed.
% \textcolor{red}{But we found a counterexample by constructing an ellipses with second eigenvalue greater than the second eigenvalue of a spherical shell with the same perimeter and with fixed internal ball.}

\section{Main result}\label{main_sec}
%We recall the notation that we will use in the following: we set $\Omega :=\Omega_0\setminus \overline{B_{r}}$, where $\Omega_0$ is an open,  bounded, simply connected set of $\mathbb{R}^n$ with Lipschitz boundary  and $B_{r}\Subset\Omega_0$ is the open ball of radius $r$ centered at the origin.
In this section we prove that the spherical shell is a local maximizer %in the $L^\infty$-topology
for the first eigenvalue of \eqref{eigSD} among nearly spherical sets with fixed volume, containing $B_{r}$, for a fixed value $r>0$. %For this reason the open set $\Omega$ will be obtained by a perturbation of its external boundary, fixing the measure of the volume.\\
Firstly, we give the definition of nearly spherical sets.

\begin{defn}%\label{NSset}
Let $n\geq 2$. An open, bounded set $\Omega_0\subset \mathbb{R}^n$ with $0\in \Omega_0$ is said a nearly spherical set parametrized by $v$ if there exists $v\in W^{1,\infty}(\mathbb{S}^{n-1})$ such that
\begin{equation}\label{nearly}
 \partial    \Omega_0= \left\{y \in \R^n \colon y=R\xi  (1+v(\xi)), \, \xi \in \mathbb{S}^{n-1}\right\},
\end{equation}
where $R$ % = (\omega \slash \omega_n)^{\frac{1}{n}}$ 
is the radius of the ball having the same measure of $\Omega_0$ and   $||v||_{W^{1,\infty}%(\mathbb{S}^{n-1})
}\le1$.
\end{defn}
The volume of a nearly spherical set is given by 
	\begin{equation*}%\label{vol_ns}
	|\Omega_0|=\dfrac{1}{n}\int_{\mathbb{S}^{n-1}} \left(1+v(\xi) \right)^n\;d\mathcal{H}^{n-1}.
	\end{equation*}
%\begin{equation}
%\label{raggio}
%r(\xi) = \rho (1+v(\xi))  \in [0, +\infty]
%\end{equation}
%where $\xi \in \mathbb{S}^{n-1}$ and $\rho = (\omega \slash \omega_n)^{\frac{1}{n}}$ is the radius of the ball having the same measure of $\Omega$.

%\begin{defn} ($\mathcal{N}(n,\varepsilon)$-functions) For every $n \in \mathbb{N}$ we denote by $\mathcal{N}(n,\varepsilon)$ the set of functions $v \in W^{1,\infty} (\mathbb{S}^{n-1})$ such that
%\begin{enumerate}
%\item $\| v \|_{W^{1,\infty}(\mathbb{S}^{n-1})} \le \varepsilon$, \\
%\item $\ds \frac{1}{n} \int_{\mathbb{S}^{n-1}} (1+v(\xi))^n \, d\mathcal{H}^{n-1} = \omega_n$ \, (\textit{Volume constraint)},
%\item $\ds \int_{\mathbb{S}^{n-1}} (1+v(\xi))^{n+1} \, d\mathcal{H}^{n-1} = 0$ \, (\textit{Barycenter constraint)},
%%where $\omega_n$ is volume of the unit ball in $\mathbb{R}^n$.
%\end{defn}
%\begin{defn} \label{NSsets}($\mathcal{N}(n,\omega,\varepsilon)$-NS sets) Let $\omega>0$ and $0\le \varepsilon <1$. An open set $\Omega_0 \subset \mathbb{R}^n$ is said $\mathcal{N}(n,\omega,\varepsilon)$-NS set if there exists a function $v\in \mathcal{N}(n,\varepsilon)$ such that, up to a translation, its boundary can be represented in polar coordinates as
%\begin{equation}
%\label{raggio}
%r(\xi) = \rho (1+v(\xi))  \in [0, +\infty]
%\end{equation}
%where $\xi \in \mathbb{S}^{n-1}$ and $\rho = (\omega \slash \omega_n)^{\frac{1}{n}}$ is the radius of the ball having the same measure of $\Omega$.
%\end{defn}
The class of nearly spherical sets has a peculiar importance in shape optimization theory, in particular for stability results for spectral inequalities. In this paper, we are considering sets $\Omega=\Omega_0\setminus \overline{B}_r$ beloging to $\mathcal{A}_r$ with $r>0$, with $\Omega_0$ nearly spherical.  Now, we are in position to state the main Theorem of this article.
\begin{thm}
\label{main_result}
Let $n\geq 2$, $r>0$, $\omega>0$ and  let  $R>r$  be such that $|A_{r,R}|=\omega$.  There exists $\varepsilon=\varepsilon(n,r,\omega)>0$ such that, for any $\Omega=\Omega_0\setminus\overline{B_{r}}$ belonging to $\mathcal{A}_r$, with $\Omega_0$ nearly spherical set parametrized by $v$ such that  $||v||_{W^{1,\infty}}\le  \varepsilon$ and $|\Omega|=\omega$, then
	\begin{equation}\label{main_ineq}
	\sigma_1 (\Omega)\leq \sigma_1(A_{r,R}).
	\end{equation}
		Moreover the equality in \eqref{main_ineq} holds if and only if $\Omega$ is a spherical shell.
%	Let  $n\geq2$, $r>0$, $\omega>0$ and $0\leq\varepsilon<1-  r(\omega_n \slash \omega)^{\frac{1}{n}}$.
%	Let $\Omega_0$ be a $\mathcal{N}(n,\omega,\varepsilon)$-NS set such that $B_{r}\Subset\Omega_0$.
%Then, if we set $\Omega=\Omega_0\setminus\overline{B_{r}}$, we have 
%	\begin{equation}\label{main_ineq}
%	\sigma_1 (\Omega)\leq \sigma_1(A_{r,R}),
%	\end{equation}
%	where $A_{r,R}=B_{R}\setminus\overline{B_{r}}$ is the spherical shell with the same volume as $\Omega$.\\
%	Moreover the equality in \eqref{main_ineq} holds if and only if $\Omega$ is a spherical shell.
	%, provided
%	$$   B_{r}\subset   \{  x\in\mathbb{R}^n\;:\; {\rm dist}(x, \mathbb{R}^n\setminus B_{R} ) >\delta  \}\subset  \Omega_0\subset \{  x\in\mathbb{R}^n\;:\; {\rm dist}(x, B_{R} ) <\delta  \}$$ 
%	and $\Omega_0$ is a nearly spherical set, whose boundary is parametrized as in Definition \ref{NSsets} by a function $v\in W^{1,\infty}(\mathcal{S}^{n-1})$ with $\| v \|_{W^{1,\infty}(\mathbb{S}^{n-1})} < \delta$,
%	 for some positive $\delta$ that depends on the volume of $\Omega$ and $n$ only. 
	\end{thm}
Let us remark that, in order to have $B_{r}\Subset\Omega_0$, we need to require that $\eps\leq 1-r/R$ to verify that $|y|\ge r$, that is $R (1+v(\xi))  \ge r$. Moreover, we observe that, since all the quantities involved are translation invariant, the result in Theorem \ref{main_result} holds also  among nearly spherical sets with fixed volume and containing a fixed internal ball.
%For what follows, considerating that the domain we consider it is not connected, it is necessary to clarify that the perturbation is done on $\Omega_0$. So $|A_{r,R}| = | \Omega_0 \slash \bar{B}_{r}|$.\\

Recalling the explicit expression \eqref{radeigfun} of the first eigenfunction $z$ on the spherical shell $A_{r,R}$, we define the \textit{weighted volume} and the \textit{weighted perimeter} as:
\begin{align*}
%\label{wei_vol}
V(\Omega)&:= \int_{\Omega} |\nabla z|^2 \,dx,\\
%\label{wei_per}
P(\Omega)&:= \int_{\partial\Omega_0} z^2 \,dx.
\end{align*} 
Furthermore, to simplify the notations, we set, for $n=2$,
\begin{align}
 \label{hrho1}
h_{R} (t)&= (\ln(tR) -\ln r)^2\\
 \label{frho1}
f_{R}(t) &= \frac{h'_{R}(t)}{2R} = \frac{\sqrt{h_{R} (t)}}{(tR)}, 
\end{align}
and for $n\ge 3$
\begin{align} 
\label{hrho}
h_{R} (t)&= \left( \dfrac{1}{r^{n-2}}-\dfrac{1}{(tR)^{n-2}}\right)^2\\
\label{frho}
f_{R}(t) &= \frac{h'_{R}(t)}{2R} = \frac{n-2}{(tR)^{n-1}} \left( \dfrac{1}{r^{n-2}}-\dfrac{1}{(tR)^{n-2}}\right),
\end{align}
where $R$ is the radius of the ball with the same volume of $\Omega_0$ and $t\ge\frac{r}R$.

Now, we write the Raylegh quotient \eqref{ray_quo} using the parametrization in \eqref{nearly}.

\begin{lem}%\label{perturbation_lemma}
Let $n\geq 2$, $r>0$, $\omega>0$ and  let  $R>r$  be such that $|A_{r,R}|=\omega$.  For any  $0<\varepsilon <1-r/R$ and for any   $\Omega=\Omega_0\setminus\overline{B_{r}}$ belonging to $\mathcal{A}_r$, with $\Omega_0$ nearly spherical set parametrized by $v$ such that $||v||_{W^{1,\infty}}\le  \varepsilon$ and $|\Omega|=\omega$, then
\begin{equation} \label{pertquot}
\sigma_1 (\Omega) \le \frac{V(\Omega)}{P(\Omega)} = \frac{\ds \int_{\mathbb{S}^{n-1}} f_{R}(1+v(\xi))(1+v(\xi))^{n-1} \, d\mathcal{H}^{n-1}}{\ds \int_{\mathbb{S}^{n-1}} h_{R}(1+v(\xi))(1+v(\xi))^{n-1}\sqrt{1+\frac{|\nabla v(\xi)|^2}{(1+v(\xi))^2}}\,d\mathcal{H}^{n-1}}.
\end{equation}
Moreover if $\Omega =  A_{r,R}$, then equality holds in \eqref{pertquot} and $\sigma_1(A_{r,R})= \, \frac{\ds f_{R}(1)}{\ds h_{R}(1)}$.
\end{lem}
 \begin{proof} From the variational characterization \eqref{minSD}  of $\sigma_1(\Omega)$, %by the definition \eqref{wei_per} and \eqref{wei_vol} of weighted volume $V(\Omega)$ and weighted perimeter $P(\Omega)$, integrating by parts,
 we have
 \begin{equation*}
 \sigma_1(\Omega) \le\dfrac{V(\Omega)}{P(\Omega)} = \frac{\ds \int_{\Omega}|\nabla z|^2 \,dx}{\ds \int_{\partial \Omega_0}z^2 \, d\mathcal{H}^{n-1}} = \frac{\ds \int_{\partial \Omega_0}\frac{\partial z}{\partial \nu} z \,d\mathcal{H}^{n-1}}{\ds \int_{\partial \Omega_0}z^2 \, d\mathcal{H}^{n-1}}.
 \end{equation*}
The conclusion follows using the change of variables in \eqref{nearly}.
 \end{proof}
We recall the following result, whose proof can be found in  \cite{Fu}.
\begin{lem} \label{fuglede} Let $n\geq 2$ and $R>0$. There exists a constant $C=C(n)>0$ such that for any $0<\varepsilon<1$ and for any $v$ parametrizing a nearly spherical set $\Omega_0$ such that $||v||_{W^{1,\infty}}\le  \varepsilon$ and $|\Omega_0|=|B_{R}|$, then
\begin{align*}
%\label{a1} 
&  \left|(1+v)^{n-1} -\left(1+(n-1)v + (n-1)(n-2)\frac{v^2}{2}\right)\right| \le C\varepsilon v^2\ \text{on}\ \mathbb{S}^{n-1},\\
%\label{a2}  
& \ds 1+ \frac{|\nabla v|^2}{2}-\sqrt{1+\frac{|\nabla v|^2}{(1+v)^2}} \le C\varepsilon \left( v^2 + |\nabla v|^2\right)\ \text{on}\ \mathbb{S}^{n-1},\\
%\label{a3}  
& \ds \left|\int_{\mathbb{S}^{n-1}} v(\xi)\,d\mathcal{H}^{n-1} + \frac{n-1}{2} \int_{\mathbb{S}^{n-1}} v^2(\xi)\,d\mathcal{H}^{n-1} \right| \le C\varepsilon \|v\|^2_{L^2}.
\end{align*}
\end{lem}
As a consequence of the analyticity of $h_R$ and $f_R$, defined in \eqref{hrho1}-\eqref{frho1}-\eqref{hrho}-\eqref{frho}, the following Lemma holds.
\begin{lem} \label{analiticity} 
Let $n\geq 2$ and $0<r<R$. There exists $K=K(n,r,R) > 0$ such that for any $0<\varepsilon < 1$ and for any $v$ parametrizing a nearly spherical set $\Omega_0$ such that $||v||_{W^{1,\infty}}\le  \varepsilon$ and $|\Omega_0|=|B_{R}|$, then
\begin{align*}
\left|h_{R}(1+v) - h_{R}(1) - h_{R}'(1)v - h_{R}''(1) \frac{v^2}{2}\right| \le K\varepsilon v^2\ \text{on}\ \mathbb{S}^{n-1},\\
\left|f_{R}(1+v) - f_{R}(1) - f_{R}'(1)v - f_{R}''(1) \frac{v^2}{2}\right| \le K\varepsilon v^2\ \text{on}\ \mathbb{S}^{n-1}.
\end{align*}
\end{lem}
Furthermore, the following Poincar\'e inequality holds.
\begin{lem} $\textit{(Poincar\'e inequality)}$ \label{poincar} Let $n\ge 2$ and $R>0$, then there exists a positive constant $C=C(n)$ such that for any $0<\varepsilon<1$ and for any function $v$ parametrizing a nearly spherical set $\Omega_0$ such that $||v||_{W^{1,\infty}}\le  \varepsilon$ and $|\Omega_0|=|B_{R}|$, then
\begin{equation*}
\|\nabla v \|^2_{L^2} \ge (n-1)(1-C\varepsilon) \|v\|^2_{L^2}.
\end{equation*}
\end{lem}
\begin{proof}
The function $v\in L^2(\mathbb S^{n-1})$ admits a harmonic expansion %(see e.g. \cite[Chap. 3]{groemer1996geometric}), 
in the sense that there exists a family of $n$-dimensional spherical harmonics $\{ H_j(\xi) \}_{j\in\N}$ such that
\[
v(\xi)=\sum_{j=0}^{+\infty}c_j H_j(\xi), \quad \xi\in\mathcal S^{n-1}\quad\text{with}\quad \|H_j\|_{L^2(\mathbb{S}^{n-1})}=1,
\]
where 
\begin{equation*}
c_j=\ds \langle v,H_j\rangle_{L^2(\mathbb{S}^{n-1})}=\int_{\mathbb S^{n-1}}v(\xi)H_j(\xi) d\mathcal H^{n-1}.
\end{equation*}
and $H_j$ satisfying
\[
\Delta_{\mathcal S^{n-1}} H_j=j(j+n-2)H_j,\quad\forall\ j\in\N,
\]
where $\Delta_{\mathcal S^{n-1}}$ is the Laplace-Beltrami operator. Furthermore the following identities % (see e.g. \cite[Th. 3.2.10]{G} and \cite[Th. 3.2.12]{G}) 
hold true
\begin{align}\label{perceval}
||v||_{L^2(\mathcal S^{n-1})}^2 & =\sum_{j=0}^\infty c_j ^2,\\%||H_j||_{L^2(\mathcal S^{n-1})} \quad 
\label{perceval_grad}
||\nabla v||_{L^2(\mathcal S^{n-1})}^2 & =\sum_{j=1}^\infty j(j+n-2) c_j ^2.
\end{align}
Since $H_0 = (n\omega_n)^{-\frac{1}{2}}$, we have
\begin{align*}
|c_0| &= (n\omega_n)^{-\frac{1}{2}} \left| \int_{\mathbb S^{n-1}}v(\xi)d\mathcal H^{n-1} \right| \le \\& (n\omega_n)^{-\frac{1}{2}} \left| \int_{\mathbb S^{n-1}}v^2(\xi)d\mathcal H^{n-1} \right|\left( \frac{n-1}{2} + C\varepsilon \right) = C\varepsilon \|v\|_{L^2},
\end{align*}
where the constant C has been renamed. Using this estimate, by \eqref{perceval} and \eqref{perceval_grad}, we have
\begin{equation*}
\|v\|_{L^2} = \sum_{j=0}^\infty c^2_j = c^2_0 + \sum_{j=1}^\infty c_j ^2 \le C\varepsilon \|v\|^2_{L^2}+\sum_{j=1}^\infty c_j ^2,
\end{equation*}
and
\begin{equation*}
\|\nabla v\|_{L^2} = \sum_{j=1}^\infty j(j+n-2) c_j ^2 \ge (n-1) \sum_{j=1}^\infty  c_j ^2 \ge (n-1)(1-C\varepsilon) \|v\|^2_{L^2},
\end{equation*}
which concludes the proof.
\end{proof}

Now we give a key estimate for the main Theorem.
\begin{prop} \label{key_prop} 
Let $n\geq 2$, $r>0$, $\omega>0$ and  let  $R>r$  be such that $|A_{r,R}|=\omega$. 	There exist two positive constants $K>0$ and $0\leq\varepsilon_0<1-r/R$, %(\omega_n \slash \omega)^{\frac{1}{n}}$ Let $\Omega_0$ be a $\mathcal{N}(n,\omega,\varepsilon)$-NS set such that $B_{r}\Subset\Omega_0$.
depending on $n$, $r$ and $\omega$ only, such that for any $0< \varepsilon < \varepsilon_0$, for any  $\Omega=\Omega_0\setminus\overline{B_{r}}$ belonging to $\mathcal{A}_r$, with $\Omega_0$ nearly spherical set parametrized by $v$ such that $||v||_{W^{1,\infty}}\le  \varepsilon$ and $|\Omega|=\omega$, then
\begin{equation}\label{key_estimate}
\begin{split}
\frac{V(\Omega^{\sharp})P(\Omega) - P(\Omega^{\sharp})V(\Omega)}{ n\omega_n} & =\\
=f_{R}(1) \ds \int_{\mathbb{S}^{n-1}} h_{R}  (1 &+v(\xi))  (1+v(\xi))^{n-1}\sqrt{1+\frac{|\nabla v(\xi)|^2}{(1+v(\xi))^2}}\,d\mathcal{H}^{n-1}  \\
-h_{R}(1) \ds \int_{\mathbb{S}^{n-1}} & f_{R} (1+v(\xi))(1+v(\xi))^{n-1} \, d\mathcal{H}^{n-1}  \ge K \int_{\mathbb{S}^{n-1}} v^2 \, d\mathcal{H}^{n-1}.
\end{split} 
\end{equation}
%where $R = (\omega \slash \omega_n)^{\frac{1}{n}}$.
\end{prop}

\begin{proof} %Let us denote by I the first quantity in \eqref{key_estimate}. 
Using Lemmata \ref{fuglede}, \ref{analiticity}, \ref{poincar}, we have
\begin{equation} \label{sviluppi}
\begin{split}
 & f_{R}(1) \ds \int_{\mathbb{S}^{n-1}} h_{R}(1+v(\xi))(1+v(\xi))^{n-1}\sqrt{1+\frac{|\nabla v(\xi)|^2}{(1+v(\xi))^2}}\,d\mathcal{H}^{n-1} \\
&\qquad - h_{R}(1) \ds \int_{\mathbb{S}^{n-1}} f_{R}(1+v(\xi))(1+v(\xi))^{n-1} \, d\mathcal{H}^{n-1}  \\
& \ge \int_{\mathbb{S}^{n-1}} v\left(f_{R}(1) h'_{R}(1)-f'_{R}(1)h_{R}(1)\right) \, d\mathcal{H}^{n-1}\\
& \qquad + \int_{\mathbb{S}^{n-1}} \frac{v^2}{2}[f_{R}(1) h''_{R}(1)-f''_{R}(1)h_{R}(1)+2(n-1)(f_{R}(1) h'_{R}(1)-f'_{R}(1)h_{R}(1))] \, d\mathcal{H}^{n-1} \\
&\qquad+\int_{\mathbb{S}^{n-1}} f_{R}(1)h_{R}(1) \frac{|\nabla v|^2}{2} d\mathcal{H}^{n-1} - \varepsilon K_1 \|\nabla v \|^2_{L^2},
\end{split}
\end{equation}
where $K_1$ is a positive constant.
Let us set
\begin{equation*}
\begin{split}
& Q_1(t) := f_{R}(t) h'_{R}(t)-f'_{R}(t)h_{R}(t), \\
& Q_2(t) := f_{R}(t) h''_{R}(t)-f''_{R}(t)h_{R}(t),\\
& Q_3(t) := f_{R}(t)h_{R}(t),
\end{split}
\end{equation*}

In order to show \eqref{key_estimate}, we need to prove
\begin{enumerate}
    \item $Q_1(1)>0$,
    \item $Q_3(1)>0$,
    \item $(n-1)\left[Q_1(1)+Q_3(1)\right] +Q_2(1)>0$.
\end{enumerate}
Indeed, when (1), (2), (3) hold, then, by using Lemmata \ref{fuglede} and \ref{poincar}, the last term in \eqref{sviluppi} can be estimated as
\begin{equation*}
\begin{split}
& Q_1(1) \int_{\mathbb{S}^{n-1}} v \, d\mathcal{H}^{n-1} + (2(n-1)Q_1(1) + Q_2(1))  \int_{\mathbb{S}^{n-1}}  \frac{v^2}{2} \, d\mathcal{H}^{n-1}\\
&\qquad\qquad\qquad\qquad\qquad\qquad\qquad+ Q_3(1) \int_{\mathbb{S}^{n-1}} \frac{|\nabla v|^2}{2} d\mathcal{H}^{n-1} - \varepsilon K_1 \|\nabla v \|^2_{L^2}\\
&\ge - \frac{n-1}{2} Q_1(1) \int_{\mathbb{S}^{n-1}} v^2 \, d\mathcal{H}^{n-1}   - \varepsilon K_2 \| v \|^2_{L^2} + \left((n-1)Q_1(1) + \frac{Q_2(1)}{2}\right) \int_{\mathbb{S}^{n-1}} v^2  \, d\mathcal{H}^{n-1} \\
&\qquad\qquad\qquad\qquad\qquad\qquad\qquad+ \frac{n-1}{2} Q_3(1) \int_{\mathbb{S}^{n-1}} v^2 - \varepsilon K_3 \| v \|^2_{L^2} - \varepsilon K_1 \|\nabla v \|^2_{L^2}\\
& =  \frac{1}{2}\left\{ (n-1)[Q_1(1)+Q_3(1)]+Q_2(1)\right\} \| v \|^2_{L^2}  \\& \qquad\qquad\qquad\qquad\qquad\qquad\qquad
- \varepsilon K_2 \| v \|^2_{L^2}  - \varepsilon K_3 \| v \|^2_{L^2} - \varepsilon K_1 \|\nabla v \|^2_{L^2} \\
& \ge K \| v \|^2_{L^2} - \varepsilon K_4 \| v \|^2_{W^{1,2}(\mathbb{S}^{n-1})},
\end{split}
\end{equation*}
where we denoted $K=\frac{1}{2}\left\{(n-1)\left[Q_1(1)+Q_3(1)\right] +Q_2(1)\right\}>0$ and $K_4 = \max \{K_1,K_2,K_3\}$.
The proof concludes by choosing $\varepsilon$ small enough.

It remains to prove (1), (2), (3) by distinguishing the bidimensional from the higher dimensional case. We note that
\begin{equation} \label{Q1}
Q_1(t) = f^2_{R}(t) \left[ \frac{h_{R}(t)}{f_{R}(t)} \right]'= 2R f^2_{R}(t) \left[ \frac{h_{R}(t)}{h'_{R}(t)} \right]',
\end{equation}
and
\begin{equation} \label{Q2}
    Q_2(t) = Q_1'(t) = \left[f^2_{R}(t)\right]'\left[ \frac{h_{R}(t)}{f_{R}(t)} \right]' + f^2_{R}(t) \left[ \frac{h_{R}(t)}{f_{R}(t)} \right]''.
\end{equation}

{\bf Case 1.} Let be $n=2$. We observe that
%\begin{equation*}
%Q_1(t) = f^2_{R}(t) \left[ \frac{h_{R}(t)}{f_{R}(t)} %\right]',
%\end{equation*}
%where
\begin{equation*}
\frac{h_{R}(t)}{f_{R}(t)} = R t (\ln(tR) -\ln r),
\end{equation*}
is positive and strictly increasing, since it is a product of two strictly increasing positive functions. Hence $Q_1(t) > 0$ and in particular 
\begin{equation*}
Q_1(1)= \frac{h_R(1)}{R}\left( \sqrt{h_R(1)}+1\right) >0.
\end{equation*}\\
Moreover, it is clear that
\begin{equation*}
Q_3(1) = \frac{h_{R}(1) \sqrt{h_{R}(1)}}{R} >0.
\end{equation*}
Let us now calculate all the terms in \eqref{Q2} and evaluate them for $t=1$. We have
%We observe that $Q_2(t) = Q'_1(t)$, so that
%\begin{equation}
%    \begin{split}
%\label{Poscost}
%2Q_1(t)+Q_2(t)&=2Q_1(t)+Q'_1(t) = 2 f^2_{R}(t) \left[ \frac{h_{R}(t)}{f_{R}(t)} \right]' + \left[f^2_{R}(t)\right]'\left[ \frac{h_{R}(t)}{f_{R}(t)} \right]'\\ & + f^2_{R}(t) \left[ \frac{h_{R}(t)}{f_{R}(t)} \right]''
%=\left[ \frac{h_{R}(t)}{f_{R}(t)} \right]'\left( 2 %f^2_{R}(t) + \left[f^2_{R}(t)\right]'\right) + %f^2_{R}(t) \left[ \frac{h_{R}(t)}{f_{R}(t)} \right]''.
%    \end{split}
%\end{equation}

\begin{align*}
& \left[ \frac{h_{R}(t)}{f_{R}(t)} \right]'_{t=1} = R \left(\sqrt{h_{R}(t)} + 1\right)_{t=1}= R \left(\sqrt{h_{R}(1)} + 1\right) >0 ,\\
& \left[ \frac{h_{R}(t)}{f_{R}(t)} \right]''_{t=1} = \left( \frac{R}{t}\right)_{t=1}= R>0  
\end{align*}
and
\begin{align*}
&  f^2_{R}(1)= \frac{h_{R}(1)}{R^2} >0,\\
&\left[f^2_{R}(t)\right]'_{t=1} = \left[\frac{2R}{(tR)^3}\left(\sqrt{h_{R}(t)} -h_{R}(t)\right)\right]= \frac{2}{R^2}\left(\sqrt{h_{R}(1)} -h_{R}(1)\right).
\end{align*}
Summing up, estimate (3) follows by 
\begin{equation*}
\begin{split}
Q_1(1)+&Q_3(1)+Q_2(1) = \frac{h_R(1)\sqrt{h_R(1)}}{R}+ \frac{h_R(1)}{R}+\frac{h_R(1)\sqrt{h_R(1)}}{R}+\\ &2\frac{\sqrt{h_R(1)}}{R}- 2\frac{h_R(1)\sqrt{h_R(1)}}{R} + \frac{h_R(1)}{R} = \frac{2}{R}(h_R(1)+\sqrt{h_R(1)}) >0.
\end{split}    
\end{equation*}

%\begin{equation*}
%\left( 2 f^2_{R}(t) %+\left[f^2_{R}(t)\right]'\right)_{t=1} = \frac{2}{R^2} %\sqrt{h_{R}(1)} > 0 .
%\end{equation*}
%Eventually
%\begin{equation*}
%2Q_1(1)+ Q_2(1) > 0.    
%\end{equation*}

 {\bf Case 2.} For $n\ge 3$, from \eqref{Q1} we have
%\begin{equation*}
%Q_1(t) = 2R f^2_{R}(t) \left[ \frac{h_{R}(t)}{h'_{R}(t)} \right]'
%\end{equation*}
%and
\begin{equation*}
\frac{h_{R}(t)}{h'_{R}(t)} = \frac{(tR)^{n-1}}{2(n-2)R} \left( \dfrac{1}{r^{n-2}}-\dfrac{1}{(tR)^{n-2}}\right),
\end{equation*}
that is a strictly increasing function, since it is product of two strictly increasing and positive functions. Hence $Q_1(t) > 0$ and, in particular
\begin{equation*}
    Q_1(1) =  \frac{(n-1)(n-2)}{R^{n-1}} h_R(1)\sqrt{h_R(1)}+ \frac{2(n-2)^2}{R^{2n-3}}h_R(1)>0.
\end{equation*}
Moreover, it is easily seen that 
\begin{equation*}
Q_3(1)= \frac{n-2}{R^{n-1}} h_R(1)\sqrt{h_R(1)}>0.
\end{equation*}
Eventually, we have
\begin{equation*}
    \begin{split}
        Q_2(1) &= \frac{(n-2)^3}{R^{3n-3}} \sqrt{h_R(1)}-\frac{(n-1)^2(n-2)}{R^{n-1}}h_R(1) \sqrt{h_R(1)} \\
        &+ \frac{(n-1)(n-2)^2}{R^{n}}h_R(1) \sqrt{h_R(1)} + \frac{(n-1)(n-2)^2}{R^{2n-2}}h_R(1) ,
    \end{split}
\end{equation*}
and therefore, it follows that $(n-1)\left[Q_1(1)+Q_3(1)\right]+Q_2(1)>0$.
%In order to prove that $2(n-1)Q_1(1) + Q_2(1) >0$ we follow the same steps as in the case of dimension 2. We have an analogue equation as in \ref{Poscost} and show that every term is positive when evaluated for $t=1$.  %, we have\begin{equation}2n Q_1(t) + Q_2 (t) = Q_1(t) \left[ 2n + \frac{Q_2(t)}{Q_1(t)} \right] = Q_1(t) \left[ 2n + (\log Q_1(t))'\right]>0,\end{equation}since $Q_1(t)$ positive and $\log Q_1(t)$ increasing.
%Moreover is easily seen that $Q_3(t) > 0$. 
\end{proof}
We use the previous result to give a stability result in a quantitative form.

\begin{thm} \label{imp}
Let $n\geq 2$, $r>0$, $\omega>0$ and  let  $R>r$  be such that $|A_{r,R}|=\omega$. 	There exist two positive constants $K>0$ and $0\leq\varepsilon_0<1-r/R$, %(\omega_n \slash \omega)^{\frac{1}{n}}$ Let $\Omega_0$ be a $\mathcal{N}(n,\omega,\varepsilon)$-NS set such that $B_{r}\Subset\Omega_0$.
depending on $n$, $r$ and $\omega$ only, such that for any $0< \varepsilon < \varepsilon_0$, for any  $\Omega=\Omega_0\setminus\overline{B_{r}}$ belonging to $\mathcal{A}_r$, with $\Omega_0$ nearly spherical set parametrized by $v$ such that $||v||_{W^{1,\infty}}\le  \varepsilon$, and $|\Omega|=\omega$, then
\begin{equation*}
\sigma_1(A_{r,R}) \ge  \sigma_1(\Omega) \left( 1+K(n,r,\omega) \int_{\mathbb{S}^{n-1}} v^2(\xi)\,d\mathcal{H}^{n-1}\right).
\end{equation*}
\end{thm}
\begin{proof}
From Proposition \ref{key_prop} we know that there exists $K>0$ such that
\begin{equation*}
P(A_{r,R})P(\Omega) \left( \frac{V(A_{r,R})}{P(A_{r,R})}-\frac{V(\Omega)}{P(\Omega)}\right) \ge n \omega_n K \int_{\mathbb{S}^{n-1}} v^2 \, d\mathcal{H}^{n-1}.
\end{equation*}
Then, we have
\begin{equation*}
\begin{split}
\sigma_1(A_{r,R}) &= \frac{V(A_{r,R})}{P(A_{r,R})} \ge \frac{V(\Omega)}{P(\Omega)} + \frac{\ds n \omega_n K \int_{\mathbb{S}^{n-1}} v^2 \, d\mathcal{H}^{n-1}}{P(A_{r,R})P(\Omega)}\\
&=\frac{V(\Omega)}{P(\Omega)} \left(1+\frac{\ds n \omega_n K \int_{\mathbb{S}^{n-1}} v^2 \, d\mathcal{H}^{n-1}}{P(A_{r,R})V(\Omega)} \right)\\
&=\frac{V(\Omega)}{P(\Omega)} \left(1+\frac{\ds K \int_{\mathbb{S}^{n-1}} v^2 \, d\mathcal{H}^{n-1}}{h_{R}(1) \ds \int_{\mathcal{S}^{n-1}} f_{R}(1+v(\xi))(1+v(\xi))^{n-1} \, d\mathcal{H}^{n-1}} \right)\\
&\ge \frac{V(\Omega)}{P(\Omega)} \left(1+\frac{\ds K \int_{\mathbb{S}^{n-1}} v^2 \, d\mathcal{H}^{n-1}}{n\omega_n 2^{n-1} h_{R}(1) f_{R}(2)} \right) \ge \sigma_1(\Omega) \left( 1+K\int_{\mathbb{S}^{n-1}} v^2 \, d\mathcal{H}^{n-1}\right),
\end{split}
\end{equation*}
where the second inequality follows by the fact that $\|v\|_{W^{1,\infty}(\mathbb{S}^{n-1})} \le \varepsilon<1$ and by the monotonicity of $f_{R}(\cdot)$. %This concludes the proof.
\end{proof}
Eventually, the main result (Theorem \ref{main_result}) easily follows by Theorem \ref{imp}. Moreover, if $\Omega =  A_{r,R}$, then the function $v$ parametrizing the outer boundary is constantly equal to zero and equality in \eqref{main_ineq} holds.
%\begin{cor} Let $\Omega_0$ be an open, bounded, connected set of $\mathbb{R}^n$ with Lipschitz boundary and $B_{r}\Subset\Omega_0$ be the open ball of radius $r>0$ centered at the origin. Denoting $\Omega :=\Omega_0\setminus \overline{B_{r}}$, for all $\omega>0$, there exist two constants $\varepsilon=\eps(n,\omega)$ and $K=K(n,\omega)$ such that if $\Omega=\Omega_0\setminus\overline{B_{r}}$ is a $(n,\omega,\varepsilon)$-NS set and $v(\xi)\in \mathcal{N}(n,\varepsilon)$ provides the polar representation of $\partial\Omega_0$, then
%\begin{equation*}
%\sigma_2(A_{r,\rho}) \ge  \sigma_2(\Omega)
%\end{equation*}
%where $\rho = (\omega \slash \omega_n)^{\frac{1}{n}}$.
%\end{cor}

\section*{Acknowledgments}
This work has been partially supported by GNAMPA of INdAM and by Progetto di eccellenza \lq\lq Sistemi distribuiti intelligenti\rq\rq\ of Dipartimento di Ingegneria Elettrica e dell'Informazione \lq\lq M. Scarano\rq\rq.

%\bibliographystyle{alpha}
%\bibliography{biblioPPS}

\end{document}